\documentclass[12pt,reqno]{amsart}
\usepackage{amsfonts}
\usepackage{amsmath,amsthm,amssymb,amsfonts,amscd}
\usepackage{mathrsfs}
\usepackage{lipsum}
\usepackage{bbding}
\usepackage{graphicx,latexsym}
\usepackage{hyperref}

\newcommand{\bea}{\begin{eqnarray}}
\newcommand{\eea}{\end{eqnarray}}
\newcommand{\bna}{\begin{eqnarray*}}
\newcommand{\ena}{\end{eqnarray*}}

\numberwithin{equation}{section} % 定义公式按节编号

\renewcommand{\thefootnote}{\fnsymbol{footnote}}
\theoremstyle{plain}
\newtheorem{theorem}{Theorem}
\newtheorem{lemma}{Lemma}

\newtheorem{proposition}{Proposition}

\theoremstyle{definition}

\newtheorem{remark}{Remark}

\newcommand\blfootnote[1]{%
  \begingroup
  \renewcommand\thefootnote{}\footnote{#1}%
  \addtocounter{footnote}{-1}%
  \endgroup
}

\begin{document}

\title{Hybrid bounds for twists of $GL(3)$ $L$-functions}

\author{Qingfeng Sun}
\address{School of Mathematics and Statistics \\ Shandong University, Weihai
\\ Weihai \\Shandong 264209 \\China}
\email{qfsun@sdu.edu.cn}

\begin{abstract}
Let $\pi$ be a Hecke-Maass cusp form for $SL(3,\mathbb{Z})$ and $\chi=\chi_1\chi_2$
a Dirichlet character with $\chi_i$ primitive modulo $M_i$. Suppose that
$M_1$, $M_2$ are primes such that
$\max\{(M|t|)^{1/3+2\delta/3},M^{2/5}|t|^{-9/20},
M^{1/2+2\delta}|t|^{-3/4+2\delta}\}(M|t|)^{\varepsilon}<M_1<
\min\{
(M|t|)^{2/5}$,\\
$(M|t|)^{1/2-8\delta}\}(M|t|)^{-\varepsilon}$ for
any $\varepsilon>0$, where $M=M_1M_2$, $|t|\geq 1$ and $0<\delta< 1/52$.
Then we have
$$
L\left(\frac{1}{2}+it,\pi\otimes \chi\right)\ll_{\pi,\varepsilon}
(M|t|)^{3/4-\delta+\varepsilon}.
$$
\end{abstract}

\keywords{Hybrid bounds, $GL(3)$ $L$-functions, twists}

\blfootnote{{\it 2010 Mathematics Subject Classification}: 11F66, 11M41}
\maketitle
\tableofcontents

\section{Introduction}

Let $\pi$ be a Hecke-Maass cusp form for $SL(3,\mathbb{Z})$ with normalized Fourier
coefficients $\lambda(n_1,n_2)$ such that $\lambda(1,1)=1$.
Let $\chi$ be a primitive
Dirichlet character modulo $M$. The $L$-function attached to the twisted form
$\pi\otimes \chi$ is given by the Dirichlet series
\bna
L(s,\pi \otimes \chi)=\sum_{n=1}^{\infty}\lambda(1,n)n^{-s}
\ena
for Re$(s)>1$, which can be continued to an entire function with
a functional equation of arithmetic conductor $M^3$. Thus by the Phragmen-Lindel\"{o}f
principle one derives the convexity bound
$L\left(1/2+it,\pi\otimes \chi\right)\ll_{\pi,\varepsilon}
\left(M(1+|t|)\right)^{3/4+\varepsilon}$,
where $\varepsilon>0$ is arbitrary.
The important challenge for us is to prove a sub-convexity bound which improves the convexity bound
by a positive constant. There has been great progress for the sub-convexity problem
of $L(s,\pi \otimes \chi)$ in the work \cite{B}, \cite{HBR} and \cite{Munshi1}-\cite{Munshi5}
(also see \cite{Li}, \cite{McKee} and \cite{Nunes} for $t$-aspect sub-convexity for
$L(s,\pi)$).
In \cite{B}, Blomer established the bound
$$L\left(\frac{1}{2}+it,\pi\otimes \chi\right)\ll_{\pi,t,\varepsilon}
M^{3/4-1/8+\varepsilon}$$
for $\pi$ self-dual and $\chi$ a quadratic character modulo prime $M$.
This was extended by Huang in \cite{HBR}, where by combining the methods in \cite{B} and \cite{Li},
he showed that
$$L\left(\frac{1}{2}+it,\pi\otimes \chi\right)\ll_{\pi,\varepsilon}
(M(1+|t|))^{3/4-1/46+\varepsilon}$$ for the same
form $\pi \otimes \chi$ as in \cite{B}.
For general $GL(3)$ Hecke-Maass cusp forms, the sub-convexity results
have recently been established in several cases by Munshi
in a series of papers \cite{Munshi31}-\cite{Munshi5}.
In the $t$-aspect, Munshi proved in \cite{Munshi32} that
\bea
L\left(\frac{1}{2}+it,\pi\right)\ll_{\pi,\varepsilon}
(1+|t|)^{3/4-1/16+\varepsilon}.
\eea
For $\chi$ a primitive Dirichlet character modulo prime $M$,
he proved in \cite{Munshi4}, \cite{Munshi5} that
$$
L\left(\frac{1}{2},\pi\otimes \chi\right)\ll_{\pi,\varepsilon}
M^{3/4-1/308+\varepsilon}.
$$
For $\chi=\chi_1\chi_2$ a Dirichlet character with $\chi_i$ primitive
modulo prime $M_i$ such that $\sqrt{M_2}M^{4\vartheta}<M_1<M_1M^{-3\vartheta}$,
he showed in \cite{Munshi31} that
$$
L\left(\frac{1}{2},\pi\otimes \chi\right)\ll_{\pi,\varepsilon}
M^{3/4-\vartheta+\varepsilon},
$$
where $M=M_1M_2$ and $0<\delta<1/28$.

The purpose of this paper is to extend Munshi's some results in \cite{Munshi31} and
\cite{Munshi32}. Our main result is the following.

\begin{theorem}
Let $\pi$ be a Hecke-Maass cusp form for $SL(3,\mathbb{Z})$ and $\chi=\chi_1\chi_2$
a Dirichlet character with $\chi_i$ primitive modulo $M_i$.
Suppose that $M_1$, $M_2$ are primes such that
$\max\{(M|t|)^{1/3+2\delta/3},
M^{2/5}|t|^{-9/20},M^{1/2+2\delta}|t|^{-3/4+2\delta}\}
(M|t|)^{\varepsilon}<M_1<
\min\{(M|t|)^{2/5}$,\\$(M|t|)^{1/2-8\delta}\}(M|t|)^{-\varepsilon}$
for any $\varepsilon>0$,
where $M=M_1M_2$, $|t|\geq 1$ and $0<\delta<1/52$. Then we have
$$
L\left(\frac{1}{2}+it,\pi\otimes \chi\right)\ll_{\pi,\varepsilon}
(M|t|)^{3/4-\delta+\varepsilon}.
$$
\end{theorem}

We also have a result compared to (1.1).
\begin{theorem}
Let $\pi$ be a Hecke-Maass cusp form for $SL(3,\mathbb{Z})$ and $\chi=\chi_1\chi_2$
a Dirichlet character with $\chi_i$ primitive modulo $M_i$.
Suppose that $M_1$, $M_2$ are primes such that
$\max\{M^{3/8-2\delta/3}|t|^{3/8},M^{2/5}
|t|^{-9/20},M^{5/8-2\delta}|t|^{-5/8}\}
(M|t|)^{\varepsilon}<M_1<
\min\{(M|t|)^{2/5},M^{8\delta}\}(M|t|)^{-\varepsilon}$ for any $\varepsilon>0$,
where $M=M_1M_2$, $|t|\geq 1$ and $0<\delta\leq 1/16$. Then we have
$$
L\left(\frac{1}{2}+it,\pi\otimes \chi\right)\ll_{\pi,\varepsilon}M^{\delta}
(M|t|)^{3/4-1/16+\varepsilon}.
$$
\end{theorem}

\begin{remark}
Theorems 1 and 2 give us a sub-convexity bound for $L\left(\frac{1}{2}+it,\pi\otimes \chi\right)$
for $M$ and $t$ in some range.
For example, if $|t|>M^{1/5}$ and $(M|t|)^{1/3+2\delta/3+\varepsilon}
<M_1<(M|t|)^{2/5-\varepsilon}$ with $0<\delta\leq 1/80$, then we have
$$
L\left(\frac{1}{2}+it,\pi\otimes \chi\right)\ll_{\pi,\varepsilon}
(M|t|)^{3/4-\delta+\varepsilon}.
$$
If $|t|>M^{1/4}$ and $(M|t|)^{3/8+\varepsilon}M^{-2\delta/3}
<M_1<M^{8\delta-\varepsilon}$ with $0<\delta\leq 1/16$, then we have
$$
L\left(\frac{1}{2}+it,\pi\otimes \chi\right)\ll_{\pi,\varepsilon}
M^{\delta}(M|t|)^{3/4-1/16+\varepsilon}.
$$
\end{remark}

To prove Theorems 1 and 2, we will use the same method as in \cite{Munshi31} and
\cite{Munshi32}.
Suppose that $t\geq 1$. Then by the approximate functional equation we have
\bea
L\left(\frac{1}{2}+it,\pi\otimes \chi\right)\ll_{\pi,\varepsilon} (M t)^{\varepsilon}
\sup_{N\leq (Mt)^{3/2+\varepsilon}}\frac{|\mathcal {S}(N)|}{\sqrt{N}},
\eea
where
\bna
\mathcal {S}(N)=\sum_{n=1}^{\infty}\lambda(1,n)\chi(n)n^{-it}
V\left(\frac{n}{N}\right)
\ena
for some smooth function $V$ supported in $[1,2]$, normalized
such that $\int_{\mathbb{R}}V(v)\mathrm{d}v=1$ and satisfying $V^{(\ell)}(x)\ll_\ell 1$.
Note that by Cauchy's inequality and the Rankin-Selberg estimate $
\sum_{n\leq x}|\lambda(1,n)|^2\ll_{\pi}x,$
we have the trivial bound $\mathcal {S}(N)\ll_{\pi,\varepsilon} N$. Thus Theorem
1 (resp. Theorem 2) is true for
$N\ll (Mt)^{3/2-2\delta}$ (resp.
$N\ll (Mt)^{11/8}M^{2\delta}$).
In the following, we will estimate $\mathcal {S}(N)$ in the range
\bea
(Mt)^{3/2-2\delta}<N\leq (Mt)^{3/2+\varepsilon}\quad(\mbox{resp.}\;
(Mt)^{11/8}M^{2\delta}<N\leq (Mt)^{3/2+\varepsilon}).
\eea

The first step is to separate the Fourier coefficients $\lambda(1,n)$ and
$\chi(n)n^{-it}$. Let $\delta(n)=1$ if $n=0$ and equals 0 otherwise.
Like in \cite{Munshi31} and
\cite{Munshi32} we apply Kloosterman's version
of the circle method, which states that for any $n\in \mathbb{Z}$ and $Q\in \mathbb{R}^+$, we have
\bea
\delta(n)=2\mathrm{Re}\int_0^1 \sum_{1\leq q\leq Q}\sum_{Q<a\leq q+Q \atop (a,q)=1}
\frac{1}{aq}e\left(\frac{n\overline{a}}{q}-\frac{n\zeta}{aq}\right)\mathrm{d}\zeta,
\eea
where throughout the paper $e(z)=e^{2\pi iz}$ and $\overline{a}$ denotes
the multiplicative inverse of $a$ modulo $q$.

To construct a conductor lowering system to take care of both
$t$-aspect and $M$-aspect, we introduce a parameter $K$ satisfying
$(Mt)^{\varepsilon}<K<t$ and write
\bna
\mathcal {S}(N)=\frac{1}{K}\int_{\mathbb{R}}
V\left(\frac{v}{K}\right)\sum_{n=1}^{\infty}
\lambda(1,n)V\left(\frac{n}{N}\right)
\sum_{m\in \mathbb{Z} \atop M_1|n-m}
\chi(m)m^{-it}U\left(\frac{m}{N}\right)
\delta\left(\frac{n-m}{M_1}\right)
\left(\frac{n}{m}\right)^{iv}\mathrm{d}v,
\ena
where $U$ is a smooth function supported in $[1/2,5/2]$, $U(x)=1$
for $x\in [1,2]$ and $U^{(\ell)}(x)\ll_\ell 1$.
Applying (1.4) and choosing
\bna
Q=\sqrt{\frac{N}{KM_1}}
\ena
we get
\bna
\mathcal {S}(N)=\mathcal {S}^+(N)+\mathcal {S}^-(N),
\ena
where
\bna
\mathcal {S}^{\pm}(N)&=&\frac{1}{K}\int_{\mathbb{R}}\int_0^1
V\left(\frac{v}{K}\right)
\sum_{n=1}^{\infty}
\lambda(1,n)n^{iv}V\left(\frac{n}{N}\right)
\sum_{m\in \mathbb{Z} \atop M_1|n-m}
\chi(m)m^{-i(t+v)}U\left(\frac{m}{N}\right)\\
&&\sum_{1\leq q\leq Q}\sum_{Q<a\leq q+Q\atop (a,q)=1}
\frac{1}{aq}e\left(\pm\frac{\overline{a}(n-m)}{qM_1}\mp\frac{(n-m)\zeta}{aqM_1}\right)
\mathrm{d}v\mathrm{d}\zeta.
\ena

In the rest of the paper we will estimate $\mathcal {S}^+(N)$ (the same analysis
holds for $\mathcal {S}^-(N)$).
Denote by $\mathcal {S}^{\flat}(N)$ and $\mathcal {S}^{\sharp}(N)$ the contribution
to $\mathcal {S}^+(N)$ from $M_1|q$ and $(M_1,q)=1$, respectively.
Then Theorems 1 and 2 follow from (1.2), (1.3) and the following propositions.

\begin{proposition}
Assume $K<\min\left\{t,NM_1/M^2\right\}(Mt)^{-\varepsilon}$.
Then we have
\bna
\mathcal {S}^{\flat}(N)\ll N\sqrt{M t}/M_1^{3/2}.
\ena

\end{proposition}

\begin{proposition}
Assume $(M t)^{6/5}/(NM_1)^{3/5}\leq
K<\min\left\{t,(Mt)^2/NM_1,NM_1/M^2\right\}(Mt)^{-\varepsilon}$.
Then we have
\bna
\mathcal {S}^{\sharp}(N)\ll \left\{\begin{array}{ll}
N^{5/8}(M t)^{1/2},
&\mbox{if} \quad (M t)^{24/17}
M_1^{8/17}<N\leq (M t)^{3/2+\varepsilon},\\
\\
N^{1/5}(M t)^{11/10}M_1^{1/5}&\mbox{if}
\quad N\leq (M t)^{24/17}
M_1^{8/17}.\\
\end{array}
\right.
\ena

\end{proposition}

For our purpose we choose the optimal $K$ as
\bea
K=\max\left\{\frac{N^{1/4}}{M_1},\frac{(M t)^{6/5}}{(NM_1)^{3/5}}\right\}.
\eea
Propositions 1 and 2 will be proved by summation formulas of Voronoi's type and
stationary phase method, which are listed in Section 2.

\begin{remark}
With $K$ as in (1.5), one sees that the assumptions for $K$
in Propositions 1 and 2 are fulfilled if $M_1$ is in the range of
Theorem 1 or Theorem 2.
\end{remark}

\begin{remark}
In the appendix of \cite{Munshi31}, Munshi showed that Kloosterman's circle
method with suitable conductor lowering mechanism also works for
$\chi$ with a prime power modulus. For hybrid bounds in $t$ and $M$
aspect, we will study this in a separate paper.
\end{remark}

%\medskip
\noindent
{\bf Notation.}
Throughout the paper, the letters $q$, $m$ and $n$, with or without subscript,
denote integers. The letter $\varepsilon$ is an arbitrarily small
positive constant, not necessarily the same at different occurrences. The symbol
$\ll_{a,b,c}$ denotes that the implied constant depends at most on $a$, $b$ and $c$.
The symbols $q\sim C$ and $q\asymp C$ mean that $C<q\leq 2C$ and $c_1C\leq q\leq c_2C$
for some absolute constants $c_1, c_2$, respectively. Finally, fractional numbers
such as $\frac{ab}{cd}$ will be written as $ab/cd$ and $a/b+c$ or $c+a/b$ means
$\frac{a}{b}+c$.

\section{Voronoi formula and stationary phase method}
\setcounter{equation}{0}
\medskip

\noindent{\bf 2.1 $GL(3)$ cusp forms and Voronoi formula.}
Let $\pi$ be a Hecke-Maass cusp form of type $\nu=(\nu_1,\nu_2)$ for $SL(3,\mathbb{Z})$,
which has a Fourier-Whittaker expansion (see \cite{G}) with Fourier coefficients $\lambda(n_1,n_2)$,
normalized so that $\lambda(1,1)=1$. By Rankin-Selberg theory,
the Fourier coefficients $\lambda(n_1,n_2)$ satisfy
\bea
\mathop{\sum\sum}_{n_1^2n_2\leq x} \left|\lambda(n_1,n_2)\right|^2\ll_{\pi,\varepsilon}
x^{1+\varepsilon}.
\eea
Let
\bna
\mu_1=-\nu_1-2\nu_2+1, \qquad \mu_2=-\nu_1+\nu_2,\qquad  \mu_3=2\nu_1+\nu_2-1.
\ena
The generalized Ramanujan conjecture asserts that $\mathrm{Re}(\mu_j)=0$, $1\leq j\leq 3$,
while the current record bound due to
Luo, Rudnick and Sarnak \cite{LRS} is
$
|\mathrm{Re}(\mu_j)| \leq 1/2-1/10, 1\leq j\leq 3.
$
For $\ell=0,1$ we define
\bna
\gamma_\ell (s)=\frac{1}{2\pi^{3(s+1/2)}}\prod_{j=1}^3
\frac{\Gamma\left((1+s+\mu_j+\ell)/2\right)}
{\Gamma\left((-s-\mu_j+\ell)/2\right)}
\ena
and set $\gamma_{\pm}(s)=\gamma_0(s)\mp i \gamma_1(s)$. Then
for $\sigma\geq -1/2$,
\bea
\gamma_{\pm}(\sigma+i\tau)\ll_{\pi,\sigma}(1+|\tau|)^{3\left(\sigma+1/2\right)},
\eea
and for $|\tau|\gg (Mt)^{\varepsilon}$, we can apply Stirling's formula to get
(see \cite{Munshi32})
\bea
\gamma_{\pm}\left(-\frac{1}{2}+i\tau\right)
=\left(\frac{|\tau|}{e\pi}\right)^{3i\tau}\Psi_{\pm}(\tau),\qquad
\mathrm{where}\qquad
\Psi'_{\pm}(\tau)\ll \frac{1}{|\tau|}.
\eea

Let $\phi(x)$ be a smooth function compactly supported on $(0,\infty)$ and
denote by $\widetilde{\phi}(s)$
the Mellin transform of $\phi(x)$.
Let
\bna
\Phi_{\phi}^{\pm}\left(x\right)
=\frac{1}{2\pi i}\int\limits_{(\sigma)}x^{-s}
\gamma_{\pm}(s)\widetilde{\phi}(-s)\mathrm{d}s,
\ena
where $\sigma>\max\limits_{1\leq j\leq 3}\{-1-\mathrm{Re}(\mu_j)\}$.
Then we have the following Voronoi-type formula (see \cite{GL1}, \cite{MS}).

\medskip

\begin{lemma}
Suppose that $\phi(x) \in C_c^\infty(0,\infty)$. Let
$a,q\in \mathbb{Z}$ with $q\geq 1$, $(a,q)=1$ and $a \overline{a} \equiv 1(\bmod q)$. Then
\bna
\sum_{n=1}^{\infty}\lambda(1,n)e\left(\frac{an}{q}\right)\phi(n)=
q\sum_{\pm}\sum_{n_1|q}
\sum_{n_2=1}^{\infty}
\frac{\lambda(n_2,n_1)}{n_1n_2}S\left(\overline{a},\pm n_2;\frac{q}{n_1}\right)
\Phi_{\phi}^{\pm}\left(\frac{n_1^2n_2}{q^3}\right),
\ena
where $S(m,n;c)$ is the classical Kloosterman sum.
\end{lemma}

\medskip

\noindent{\bf 2.2 Exponential integral and stationary phase method.}

Here we collect relevant results from \cite{BHY}, \cite{Hux}, \cite{Munshi32} and \cite{P}
that will be used to estimate
some exponential integrals in this paper. First we need the stationary phase estimates from
\cite{Hux} which will be used to derive asymptotic expansion of the exponential integral
\bna
\mathcal {I}=\int_a^b g(v)e(f(v))\mathrm{d}v,
\ena
where $f$, $g$ are smooth real
valued functions and $\mathrm{Supp}(g)\subset [a,b]$.

\begin{lemma}
Assume that $\Theta_f, \Omega_f\gg b-a$ and
\bea
f^{(i)}(v)\ll \Theta_f\Omega_f^{-i}, \qquad g^{(j)}(v)\ll \Omega_g^{-j}
\eea
for $i=2,3$ and $j=0,1,2$.

(1) Suppose $f'$ and $f''$ do not vanish in $[a,b]$.
Let $\Lambda=\min\limits_{[a,b]}|f'(v)|$. Then we have
\bna
\mathcal {I}\ll \frac{\Theta_f}{\Omega_f^2\Lambda^3}\left(1+\frac{\Omega_f}{\Omega_g}
+\frac{\Omega_f^2}{\Omega_g^2}\frac{\Lambda}{\Theta_f/\Omega_f}\right).
\ena

(2) Suppose $f'$ changes sign from negative to positive at the unique point
$v_0\in (a,b)$. Let $\kappa=\min\{b-v_0,v_0-a\}$. Further suppose
(2.4) holds for $i=4$ and
\bna
f^{(2)}(v)\gg \Theta_f/\Omega_f^2.
\ena
Then
\bna
\mathcal {I}=\frac{g(v_0)e\left(f(v_0)+1/8\right)}{\sqrt{f''(v_0)}}
+O\left(\frac{\Omega_f^4}{\Theta_f^2\kappa^3}+
\frac{\Omega_f}{\Theta_f^{3/2}}
+\frac{\Omega_f^3}{\Theta_f^{3/2}\Omega_g^2}\right).
\ena
\end{lemma}

%Next we need the following second derivative bound for exponential integrals in two variables
%(see \cite{Munshi32}).
%\begin{lemma} Suppose $f$, $g$, $r_1$ and $r_2$ satisfy
%\bna
%f^{(2,0)}(x,y)\gg r_1^2, \qquad  f^{(0,2)}(x,y)\gg r_2^2
%\ena
%\bna
%f^{(2,0)}(x,y)f^{(0,2)}(x,y)-\left[f^{(1,1)}(x,y)\right]^2\gg r_1^2r_2^2.
%\ena
%Then we have
%\bna
%\int_a^b\int_c^d g(x,y)e(f(x,y))\mathrm{d}x\mathrm{d}y\ll \frac{\mathrm{Var}(g)}{r_1r_2}
%\ena
%where
%\bna
%\mathrm{Var}(g)=\int_a^b\int_c^d
%\left|\frac{\partial^2}{\partial x\partial y}g(x,y)\right|\mathrm{d}x\mathrm{d}y.
%\ena
%\end{lemma}

For the special exponential integral
\bna
U^\dag\left(r,s\right)=\int_0^{\infty}U(x)e(-rx)x^{s-1}\mathrm{d}x
\ena
where $U$ is a smooth real valued function with
$\mathrm{Supp}(U)\subset [a,b]\subset (0,\infty)$,
we quote the following result from \cite{Munshi32}.
\begin{lemma}
Suppose $U^{(j)}(x)\ll_{a,b,j} 1$.
Let $r\in \mathbb{R}$ and $s=\sigma+i\beta\in \mathbb{C}$. We have
\bea
U^\dag\left(r,s\right)=\frac{\sqrt{2\pi}e\left(1/8\right)}{\sqrt{-\beta}}
U\left(\frac{\beta}{2\pi r}\right)\left(\frac{\beta}{2\pi r}\right)^{\sigma}
\left(\frac{\beta}{2\pi er}\right)^{i\beta}+
O\left(\min\{|\beta|^{-3/2},|r|^{-3/2}\}\right),
\eea
where the implied constant depends only on $a$, $b$ and $\sigma$.
We also have
\bea
U^\dag\left(r,s\right)\ll_{a,b,\sigma,j}
\min\left\{\left(\frac{1+|\beta|}{|r|}\right)^j,\left(\frac{1+|r|}{|\beta|}\right)^j\right\}.
\eea
\end{lemma}
In applications, the $O$-term in (2.5) is not essential.
For our purpose, we will also use the following more precise asymptotic
expansion to simplify computation (see \cite{BHY}, Proposition 8.2). For a proof, see also \cite{P}.

\begin{lemma}
Let $r\in \mathbb{R}$ and $s=\sigma+i\beta\in \mathbb{C}$ such that
$x_0=\beta/(2\pi r)\in [a/2,2b]$. Then we have
\bea
U^{\dag}(r,s)=\frac{\sqrt{2\pi}e\left(1/8\right)}{\sqrt{-\beta}}
U^*\left(\frac{\beta}{2\pi r}\right)\left(\frac{\beta}{2\pi r}\right)^{\sigma}
\left(\frac{\beta}{2\pi er}\right)^{i\beta}+
O\left(\min\{|\beta|^{-5/2},|r|^{-5/2}\}\right),
\eea
where $
U^*(x_0)=x_0^{1-\sigma}\sum_{n=0}^5p_n(x_0)$
and
\bna
p_n(x_0)=\frac{1}{n!}\left(\frac{i}{2h''(x_0)}\right)^nG^{(2n)}(x_0).
\ena
Here $h(x)=-2\pi rx+\beta \log x$, $G(x)=U(x)x^{\sigma-1}e^{iH(x)}$ and
\bna
H(x)=h(x)-h(x_0)-\frac{1}{2!}h''(x_0)(x-x_0)^2.
\ena
Moreover, $G^{(2n)}(x_0)$ is a linear combination of terms of the form
$(U(x)x^{\sigma-1})^{(\ell_0)}|_{x=x_0}H^{(\ell_1)}(x_0)\cdot\cdot\cdot H^{(\ell_i)}(x_0)$,
where $\ell_0+\ell_1+\cdot\cdot\cdot+\ell_i=2n$, so that $
U^{*(\ell)}(x_0)\ll_{\sigma,a,b,\ell} 1.$
\end{lemma}

\section{Estimating $\mathcal {S}^{\flat}(N)$}
\setcounter{equation}{0}
\medskip

Recall that
\bna
\mathcal {S}^{\flat}(N)&=&\frac{1}{K}\int_{\mathbb{R}}\int_0^1
V\left(\frac{v}{K}\right)
\sum_{n=1}^{\infty}
\lambda(1,n)n^{iv}V\left(\frac{n}{N}\right)
\sum_{1\leq q\leq Q/M_1}\sum_{Q<a\leq qM_1+Q\atop (a,qM_1)=1}
\frac{1}{aqM_1}\\
&&e\left(\frac{\overline{a}n}{qM_1^2}-\frac{n\zeta}{aqM_1^2}\right)
\sum_{m\in \mathbb{Z} \atop M_1|n-m}
\chi(m)m^{-i(t+v)}U\left(\frac{m}{N}\right)
e\left(-\frac{\overline{a}m}{qM_1^2}+\frac{m\zeta}{aqM_1^2}\right)
\mathrm{d}v\mathrm{d}\zeta.
\ena
Applying Poisson summation formula with modulus
$qM_1^2M_2$ on the sum over $m$ we get
\bna
&&\sum_{m\in \mathbb{Z} \atop M_1|n-m}
\chi(m)m^{-i(t+v)}U\left(\frac{m}{N}\right)
e\left(-\frac{\overline{a}m}{qM_1^2}+\frac{m\zeta}{aqM_1^2}\right)\\
&=&\frac{N^{1-i(t+v)}}{qM_1^2M_2}\sum_{m\in \mathbb{Z}}
\mathscr{E}(a,m,q)U^{\dag}\left(\frac{N(ma-\zeta M_2)}{aqM_1^2M_2},1-i(t+v)\right),
\ena
where $U^{\dag}(r,s)=\int_0^{\infty}U(y)e(-ry)y^{s-1}\mathrm{d}y$
and
\bna
\mathscr{E}(a,m,q)=\sum_{c\bmod qM_1^2M_2\atop c\equiv n \bmod M_1}\chi(c)
e\left(\frac{(m-M_2\overline{a})c}{qM_1^2M_2}\right).
\ena

\begin{lemma}
Let $q=q_0M_1^jM_2^k$, $(q_0,M_1M_2)=1$ with $j,k\geq 0$.
We have
\bna
\mathscr{E}(a,m,q)
=\varepsilon_2qM_1\sqrt{M_2}
\chi_1(q_0M_2^{k+1}n)\chi_2(q_0M_1\overline{m^*})
e\left(m^*M_2^kn/M_1\right)
\ena
if $m\equiv M_2\overline{a}\bmod qM_1$, and is zero otherwise.
Here $\varepsilon_2\sqrt{M_2}$ is the value of the Gauss sum corresponding
to the character $\chi_2$, and
\bna
m^*=(m-M_2\overline{a})/M_1^{j+1}M_2^k.
\ena
In particular, we have $a\equiv \overline{m}M_2\bmod qM_1$ if $k=0$. If $k\geq 1$, we have
$M_2|m$ and
$a\equiv \overline{(m/M_2)} \bmod qM_1/M_2$.
\end{lemma}

\begin{proof}
We have
\bna
\mathscr{E}(a,m,q)
&=&\sum_{c_1\bmod q_0}e\left(\frac{
(m-M_2\overline{a})c_1}{q_0}\right)
\sum_{c_2\bmod M_1^{j+2}\atop c_2\equiv n\bmod M_1}\chi_1(q_0M_2^{k+1}c_2)
e\left(\frac{(m-M_2\overline{a})c_2}{M_1^{j+2}}\right)\\
&&\sum_{c_3\bmod M_2^{k+1}}\chi_2(q_0M_1^{j+2}c_3)
e\left(\frac{(m-M_2\overline{a})c_3}{M_2^{k+1}}\right),
\ena
where the first sum vanishes unless $m\equiv M_2\overline{a}\bmod q_0$
in which case it is $q_0$. The second sum vanishes unless $m\equiv M_2\overline{a}\bmod M_1^{j+1}$
in which case it equals
\bna
\chi_1(q_0M_2^{k+1}n)e\left(\frac{m^*M_2^kn}{M_1}\right)M_1^{j+1},
\ena
where $m^*=(m-M_2\overline{a})/M_1^{j+1}M_2^k$. Finally, the last sum
equals
\bna
\varepsilon_2\chi_2(q_0M_1)\overline{\chi_2}(m^*)M_2^k\sqrt{M_2}
\ena
if $m\equiv M_2\overline{a}\bmod M_2^{k}$, and is zero otherwise,
where $\varepsilon_2\sqrt{M_2}$ is the value of the Gauss sum corresponding
to the character $\chi_2$.
\end{proof}
Note that if $m=0$, we have $k\geq 1$ and $(m,qM_1)=M_2$. Then
\bna
\frac{N|0-\zeta M_2|}{aqM_1^2M_2}\leq \frac{N}{QM_2M_1^2}<(M t)^{-\varepsilon}t.
\ena
For $|m|\geq 1$, we have
\bna
\frac{N|ma-\zeta M_2|}{aqM_1^2M_2}\asymp \frac{N|m|}{qM_1^2M_2}.
\ena
Applying (2.6) one sees that the contribution from $m=0$ and
$|m|\geq qM_1(Mt)^{1+\varepsilon}/N$ is negligibly small. For smaller nonzero $m$,
by the second derivative bound for the exponential integral, we have
\bna
U^{\dag}\left(\frac{N(ma-\zeta M_2)}{aqM_1^2M_2},1-i(t+v)\right)\ll t^{-1/2}.
\ena
Therefore,
\bna
\mathcal {S}^{\flat}(N)&\ll&\frac{N}{M_1\sqrt{M_2t}}
\sum_{n\leq 2N}
|\lambda(1,n)|
\sum_{1\leq q\leq Q/M_1\atop (q,M_2)=1}
\frac{1}{QqM_1}\frac{qM_1(M t)^{1+\varepsilon}}{N}\\
&&+\frac{N}{M_1\sqrt{M_2t}}
\sum_{n\leq 2N}
|\lambda(1,n)|\sum_{1\leq q\leq Q/M_1\atop M_2|q}
\frac{M_2}{QqM_1}\frac{qM_1(M t)^{1+\varepsilon}}{N}\\
&\ll&N\sqrt{M t}/M_1^{3/2}.
\ena
This completes the proof of Proposition 1.

\section{Estimating $\mathcal {S}^{\sharp}(N)$--I}
\setcounter{equation}{0}
\medskip

First we detect the congruence $m\equiv n \bmod M_1$ using exponential sums to get
\bna
\mathcal {S}^{\sharp}(N)=\mathcal {S}_0(N)+\mathcal {S}_1(N)
\ena
where
\bna
\mathcal {S}_0(N)&=&\frac{1}{KM_1}\int_{\mathbb{R}}\int_0^1
V\left(\frac{v}{K}\right) \sum_{1\leq q\leq Q\atop (q,M_1)=1}\sum_{Q<a\leq q+Q\atop (a,q)=1}
\frac{1}{aq}\nonumber\\
&&\sum_{n=1}^{\infty}
\lambda(1,n)e\left(\frac{\overline{aM_1}n}{q}\right)
n^{iv}
V\left(\frac{n}{N}\right)e\left(-\frac{n\zeta}{aqM_1}\right)\nonumber\\
&&\sum_{m\in \mathbb{Z}}
\chi(m)e\left(\frac{-\overline{aM_1}m}{q}\right)
m^{-i(t+v)}U\left(\frac{m}{N}\right)e\left(\frac{m\zeta}{aqM_1}\right)
\mathrm{d}v\mathrm{d}\zeta
\ena
and
\bea
\mathcal {S}_1(N)&=&\frac{1}{KM_1}\int_{\mathbb{R}}\int_0^1
V\left(\frac{v}{K}\right) \sum_{1\leq q\leq Q\atop (q,M_1)=1}
\sum_{Q<a\leq q+Q\atop (a,q)=1}\;
\sideset{}{^*}\sum_{b\bmod M_1}\frac{1}{aq}\nonumber\\
&&\sum_{n=1}^{\infty}
\lambda(1,n)e\left(\frac{(\overline{aM_1}M_1+bq)n}{qM_1}\right)
n^{iv}
V\left(\frac{n}{N}\right)e\left(-\frac{n\zeta}{aqM_1}\right)\nonumber\\
&&\sum_{m\in \mathbb{Z}}
\chi(m)e\left(\frac{-(\overline{aM_1}M_1+bq)m}{qM_1}\right)
m^{-i(t+v)}U\left(\frac{m}{N}\right)e\left(\frac{m\zeta}{aqM_1}\right)
\mathrm{d}v\mathrm{d}\zeta,
\eea
where the $*$ denotes the condition $(b,M_1)=1$.
In the rest of the paper, we will estimate $\mathcal {S}_1(N)$.
The analysis for $\mathcal {S}_0(N)$ is similar,
and by following the proof for $\mathcal {S}_1(N)$, one can see that it is smaller.

Applying Poisson
summation with modulus $qM_1M_2=qM$ on the sum over $m$ in (4.1) we get
\bna
&&\sum_{m\in \mathbb{Z}}
\chi(m)e\left(\frac{-(\overline{aM_1}M_1+bq)m}{qM_1}\right)
m^{-i(t+v)}U\left(\frac{m}{N}\right)e\left(\frac{m\zeta}{aqM_1}\right)\nonumber\\
&=&\frac{N^{1-i(t+v)}}{qM}\sum_{m\in \mathbb{Z}}
\mathscr{D}(a,b,m,q)U^{\dag}\left(\frac{N(ma-\zeta M_2)}{aqM},1-i(t+v)\right),
\ena
where
\bna
\mathscr{D}(a,b,m,q)=\sum_{c\bmod qM}\chi(c)
e\left(\frac{cm}{qM}-\frac{c(\overline{aM_1}M_1+bq)}{qM_1}\right).
\ena

\begin{lemma}
Let $q=q_0M_2^k$, $(q_0,M_2)=1$ with $k\geq 0$.
We have
\bna
\mathscr{D}(a,b,m,q)
=\varepsilon_1\varepsilon_2q\sqrt{M}\chi_2(q_0M_1)
\overline{\chi_1}(\overline{qM_2}m-b)
\overline{\chi_2}(m_0)
\ena
if $m\equiv M_2\overline{a}\bmod q$, and is zero otherwise.
Here $\varepsilon_i\sqrt{M_i}$ is the value of the Gauss sum corresponding
to the character $\chi_i$, and $
m_0=(m-M_2\overline{a})/M_2^k.$
In particular, we have $a\equiv \overline{m}M_2$ if $k=0$. If $k\geq 1$, we have
$M_2|m$ and
$a\equiv \overline{(m/M_2)} \bmod q/M_2$.
\end{lemma}

\begin{proof}
Note that
\bna
\mathscr{D}(a,b,m,q)
&=&\sum_{c_1\bmod q_0}e\left(\frac{(m-M_2\overline{a})c_1}{q_0}\right)
\sum_{c_2\bmod M_2^{k+1}}\chi_2(q_0M_1c_2)
e\left(\frac{(m-M_2\overline{a})c_2}{M_2^{k+1}}\right)\\
&&\sum_{c_3\bmod M_1}\chi_1(q_0M_2^{k+1}c_3)
e\left(\frac{(\overline{qM_2}m-b)q_0M_2^{k+1}c_3}{M_1}\right),
\ena
where the first sum vanishes unless $m\equiv M_2\overline{a}\bmod q_0$
in which case it is $q_0$. The second sum equals $\varepsilon_2
\chi_2(q_0M_1)\overline{\chi_2}(m_0)M_2^k \sqrt{M_2}$ with $m_0=(m-M_2\overline{a})/M_2^k$
if $m\equiv M_2\overline{a}\bmod M_2^k$, and is zero otherwise.
Here $\varepsilon_i\sqrt{M_i}$ is the value of the Gauss sum corresponding
to the character $\chi_i$.
Thus the lemma follows.
\end{proof}

By Lemma 6 one sees that if $m=0$, then $k\geq 1$ and
$(m,q)=M_2$. So
\bna
\frac{N|0-\zeta M_2|}{aqM}\leq \frac{N}{QM}<(Mt)^{-\varepsilon}t.
\ena
For $|m|\geq 1$, we have
\bna
\frac{N|ma-\zeta M_2|}{aqM}\asymp \frac{N|m|}{qM}.
\ena
Applying (2.6) it follows that the contribution from $m=0$ and
$|m|\geq q(Mt)^{1+\varepsilon}/N$ is negligibly small.
For $1\leq |m|<q(Mt)^{1+\varepsilon}/N$, we have $N/(Mt)^{1+\varepsilon}<q\leq Q$.
Taking a dyadic subdivision for the sum over $q$, we have the following.

\begin{lemma}
Suppose $K<\min\left\{t,NM_1/M^2\right\}(Mt)^{-\varepsilon}$.
We have
\bna
\mathcal {S}_1(N)=\varepsilon_1\varepsilon_2\chi_2(M_1)N^{-it}
\sum_{N/(Mt)^{1+\varepsilon}<C\leq Q\atop C \; \mathrm{dyadic}}
\mathcal {S}_1(N,C)+O((Mt)^{-1000}),
\ena
where
\bna
\mathcal {S}_1(N,C)
&=&\frac{N}{KM_1\sqrt{M}}
\int_{\mathbb{R}}\int_0^1
V\left(\frac{v}{K}\right)N^{-iv} \sum_{q=q_0M_2^k\sim C\atop(q_0,M)=1}\frac{\chi_2(q_0)}{q}
\sum_{Q<a\leq q+Q\atop (a,q)=1}\frac{1}{a}\quad\sideset{}{^*}\sum_{b\bmod M_1}\nonumber\\
&&\times\sum_{1\leq |m|\leq q(Mt)^{1+\varepsilon}/N\atop m\equiv M_2\overline{a}\bmod q}
\overline{\chi_1}(\overline{qM_2}m-b)
\overline{\chi_2}(m_0)U^{\dag}\left(\frac{N(ma-\zeta M_2)}{aqM},1-i(t+v)\right)\nonumber\\
&&\times\sum_{n=1}^{\infty}
\lambda(1,n)e\left(\frac{(\overline{aM_1}M_1+bq)n}{qM_1}\right)
n^{iv}V\left(\frac{n}{N}\right)e\left(-\frac{n\zeta}{aqM_1}\right)
\mathrm{d}v\mathrm{d}\zeta.
\ena
\end{lemma}
Applying Lemma 1 with
$\phi(y)=y^{iv}V\left(y/N\right)e\left(-\zeta y/aqM_1\right)$ we have
\bna
&&\sum_{n=1}^{\infty}
\lambda(1,n)e\left(\frac{(\overline{aM_1}M_1+bq)n}{qM_1}\right)
n^{iv}
V\left(\frac{n}{N}\right)e\left(-\frac{n\zeta}{aqM_1}\right)\nonumber\\
&=&qM_1N^{iv}\sum_{\pm}\sum_{n_1|qM_1}
\sum_{n_2=1}^{\infty}
\frac{\lambda(n_2,n_1)}{n_1n_2}
S\left(\overline{\overline{aM_1}M_1+bq},\pm n_2;\frac{qM_1}{n_1}\right)
\mathcal {J}_{\pm}\left(\frac{n_1^2n_2}{q^3M_1^3},\frac{\zeta}{aqM_1}\right),\nonumber\\
\ena
where
\bna
\mathcal {J}_{\pm}\left(x,y\right)
=\frac{1}{2\pi i}\int\limits_{(\sigma)}\left(Nx\right)^{-s}
\gamma_{\pm}(s)V^{\dag}\left(Ny,-s+iv\right)\mathrm{d}s.
\ena
By (2.6),
\bna
V^{\dag}\left(\frac{\zeta N}{aqM_1},-s+iv\right)\ll_j
\min\left\{1,\left(\frac{1}{q|v-\tau|}\sqrt{\frac{N K}{M_1}}\right)^j\right\}
\ena
for any $j\geq 0$. Then shifting the contour to $\sigma=\ell$ (a large positive integer) and
taking $j=3\ell+3$ (in view of (2.2)) one has
\bna
\mathcal {J}_{\pm}\left(\frac{n_1^2n_2}{q^3M_1^3},\frac{\zeta}{aqM_1}\right)
\ll \left(\frac{1}{q}\sqrt{\frac{N K}{M_1}}\right)^{5/2}
\left(\frac{n_1^2n_2}{N^{1/2}K^{3/2}M_1^{3/2}}\right)^{-\ell}.
\ena
Thus the contribution from $n_1^2n_2\geq
N^{1/2+\varepsilon}K^{3/2}M_1^{3/2}$ is negligible. For
$n_1^2n_2< N^{1/2+\varepsilon}K^{3/2}M_1^{3/2}$,
we shift the contour to $\sigma=-1/2$, and obtain
\bna
\mathcal {J}_{\pm}\left(\frac{n_1^2n_2}{q^3M_1^3},\frac{\zeta}{aqM_1}\right)
&=&\sum_{J \in \mathscr{J}}\frac{1}{2\pi}
\int_{\mathbb{R}}
\left(\frac{Nn_1^2n_2}{q^3M_1^3}\right)^{1/2-i\tau}
\gamma_{\pm}\left(-\frac{1}{2}+i\tau\right)\\
&&V^{\dag}\left(\frac{N \zeta}{aqM_1},\frac{1}{2}+i(v-\tau)\right)W_J(\tau)\mathrm{d}\tau
+O((Mt)^{-1000}),
\ena
where as in \cite{Munshi32}, $\mathscr{J}$ is a collection of
$O(\log (Mt))$ many real numbers in the interval
$\left[-(Mt)^{\varepsilon}C^{-1}\sqrt{NK/M_1},
(Mt)^{\varepsilon}C^{-1}\sqrt{NK/M_1}\right]$, and
$W_J$ is a smooth partition of unity such that
for $J=0$, the function
$W_0(x)$ is supported in $[-1,1]$ and satisfies
$W_0^{(\ell)}(x)\ll_\ell 1$, for each $J>0$ (resp. $J<0$), the function
$W_J(x)$ is supported in $[J,4J/3]$ (resp. $[4J/3,J]$) and satisfies
$y^\ell W_J^{(\ell)}(x)\ll_\ell 1$ for all $\ell\geq 0$, and finally
\bna
\sum_{J\in \mathscr{J}}W_J(x)=1, \quad \mbox{for}\quad
x\in \left[-\frac{(Mt)^{\varepsilon}}{C}\sqrt{\frac{NK}{M_1}},
\frac{(Mt)^{\varepsilon}}{C}\sqrt{\frac{NK}{M_1}}\right].
\ena

We conclude with the following.

\begin{lemma} Let $K$ be as in Lemma 7. We have
\bna
\mathcal {S}_1(N,C)
=\sum_{1\leq L<N^{1/2+\varepsilon}K^{3/2}M_1^{3/2}
\atop L \,\mathrm{dyadic}}\sum_{J \in \mathscr{J}}\sum_{\pm}
\mathcal {S}_1(N,C,L,J,\pm)+O\left((Mt)^{-100}\right),
\ena
where
\bna
\mathcal {S}_1(N,C,L,J,\pm)
&=&\frac{N^{3/2}}{\sqrt{MM_1^3}}
\mathop{\sum\sum}_{n_1^2n_2\sim L}\frac{\lambda(n_2,n_1)}{\sqrt{n_2}}
\mathop{\sum_{q=q_0M_2^k \sim C\atop(q_0,M)=1}}_{n_1|qM_1}
\frac{\chi_2(q_0)}{q^{3/2}}\sum_{Q<a\leq q+Q\atop (a,q)=1}\frac{1}{a}\nonumber\\
&&\sum_{1\leq |m|\leq q(Mt)^{1+\varepsilon}/N\atop m\equiv M_2\overline{a}\bmod q}
\overline{\chi_2}(m_0)
\mathscr{B}(n_1,\pm n_2,m,a,q)\mathcal {J}_{J,\pm}^{*}(q,m,n_1^2n_2),
\ena
where
\bea
\mathscr{B}(n_1,n_2,m,a,q)=\sideset{}{^*}\sum_{b\bmod M_1}
\overline{\chi_1}(\overline{qM_2}m-b)
S\left(\overline{\overline{aM_1}M_1+bq},n_2;\frac{qM_1}{n_1}\right)
\eea
and
\bea
\mathcal {J}_{J,\pm}^{*}(q,m,y)=\frac{1}{2\pi}\int_{\mathbb{R}}
\left(\frac{Ny}{q^3M_1^3}\right)^{-i\tau}
\gamma_{\pm}\left(-\frac{1}{2}+i\tau\right)\mathcal {J}^{**}(q,m,\tau)
W_J(\tau)\mathrm{d}\tau
\eea
with
\bea
\mathcal {J}^{**}(q,m,\tau)&=&\int_{\mathbb{R}}\int_0^1
V\left(v\right)V^{\dag}\left(\frac{N \zeta}{aqM_1},\frac{1}{2}+i(Kv-\tau)\right)\nonumber\\
&&U^{\dag}\left(\frac{N(ma-\zeta M_2)}{aqM},1-i(t+Kv)\right)\mathrm{d}v\mathrm{d}\zeta.
\eea
\end{lemma}

Let $n=n_1'l$, $(n_1',M_1)=1$ and $l|M_1$. For $M_1=1$, by Weil's bound
for Klooertman sums $\mathscr{B}(n_1'M_1,n_2,m,a,q)\ll (q/ n_1)^{1/2}.$
Trivially, we have $
\mathcal {J}_{J,\pm}^{*}(q,m,n_1'^2M_1^2n_2)\ll C^{-1}\sqrt{NK/M_1t}.$
Thus the contribution from $l=M_1$ to $\mathcal {S}_1(N,C,L,J,\pm)$
is at most $N^{3/4}K^{7/4}(Mt)^{1/2}M_1^{-5/4}$ which is admissible
by the range of $M_1$. For $l=1$, we will need extra cancelation from
the character sum $\mathscr{B}(n_1,n_2,m,a,q)$ and
the integral $\mathcal {J}_{J,\pm}^{*}(q,m,n_1'^2M_1^2n_2)$.
Then the rest of the paper is devoted to estimating
\bea
\mathcal {S}_1^*(N,C,L,J,\pm)&=&\frac{N^{3/2}}{\sqrt{MM_1^3}}
\mathop{\sum\sum}_{n_1^2n_2\sim L}\frac{\lambda(n_2,n_1)}{\sqrt{n_2}}
\mathop{\sum_{q=q_0M_2^k \sim C\atop(q_0,M)=1}}_{n_1|q}
\frac{\chi_2(q_0)}{q^{3/2}}\sum_{Q<a\leq q+Q\atop (a,q)=1}\frac{1}{a}\nonumber\\
&&\sum_{1\leq |m|\leq q(Mt)^{1+\varepsilon}/N\atop m\equiv M_2\overline{a}\bmod q}
\overline{\chi_2}(m_0)
\mathscr{B}(n_1,\pm n_2,m,a,q)\mathcal {J}_{J,\pm}^{*}(q,m,n_1^2n_2).
\eea

\section{A decomposition of the integral $\mathcal {J}^{**}(q,m,\tau)$}
\setcounter{equation}{0}
\medskip

The aim of this section is to give a decomposition of $\mathcal {J}^{**}(q,m,\tau)$
for $|\tau|\leq (Mt)^{\varepsilon}C^{-1}\sqrt{NK/M_1}$.

\noindent{\bf 5.1. Stationary phase expansion for $U^{\dag}$ and $V^{\dag}$.}
Applying (2.7) we get
\bna
&&U^{\dag}\left(\frac{N(ma-\zeta M_2)}{aqM},1-i(t+Kv)\right)
=\frac{e\left(1/8\right)}{\sqrt{2\pi}}\frac{aqM\sqrt{t+Kv}}{N(\zeta M_2-ma)}
U^*\left(\frac{(t+Kv)aqM}{2\pi N(\zeta M_2-ma)}\right)\nonumber\\
&&\qquad\qquad
\left(\frac{(t+Kv)aqM}{2\pi e N(\zeta M_2-ma)}\right)^{-i(t+Kv)}
+O\left(t^{-5/2}\right).
\ena
By (2.5) we have
\bna
&&V^{\dag}\left(\frac{N \zeta}{aqM_1},\frac{1}{2}+i(Kv-\tau)\right)\nonumber\\
&=&\frac{e\left(1/8\right)}{\sqrt{\tau-Kv}}
V\left(\frac{(Kv-\tau)aqM_1}{2\pi N \zeta}\right)
\left(\frac{(Kv-\tau)aqM_1}{N \zeta}\right)^{1/2}
\left(\frac{(Kv-\tau)aqM_1}{2\pi e N \zeta}\right)^{i(Kv-\tau)}\nonumber
\\&&+O\left(\min\left\{|Kv-\tau|^{-3/2},
\left(\frac{N \zeta}{aqM_1}\right)^{-3/2}\right\}\right).
\ena
Plugging the above asymptotic expansions into (4.4) we obtain
\bea
\mathcal {J}^{**}(q,m,\tau)
&=&c_1M_2
\left(\frac{aqM_1}{N}\right)^{3/2}
\int_{\mathbb{R}}\int_0^1V\left(v\right)
\frac{\sqrt{t+Kv}}{\zeta^{1/2}(\zeta M_2-ma)}U^*\left(\frac{(t+Kv)aqM}{2\pi N(\zeta M_2-ma)}\right)
\nonumber\\
&&\left(\frac{(t+Kv)aqM}{2\pi e N(\zeta M_2-ma)}\right)^{-i(t+Kv)}
V\left(\frac{(Kv-\tau)aqM_1}{2\pi N \zeta}\right)\nonumber\\
&&
\left(\frac{(Kv-\tau)aqM_1}{2\pi e N \zeta}\right)^{i(Kv-\tau)}\mathrm{d}v\mathrm{d}\zeta
+O\left(t^{-5/2}+E^{**}\right)
\eea
for some absolute constant $c_1$, where
\bna
E^{**}=\frac{1}{\sqrt{t}}
\int_0^1\int_1^2\min\left\{|Kv-\tau|^{-3/2},
\left(\frac{N \zeta}{aqM_1}\right)^{-3/2}\right\}\mathrm{d}v\mathrm{d}\zeta.
\ena
To estimate the error term $E^{**}$, we split the integral over $v$
into two pieces: $|Kv-\tau|<N \zeta/aqM_1$ and $|Kv-\tau|\geq N \zeta/aqM_1$
as in \cite{Munshi32} to get
\bna
E^{**}\ll \frac{(M t)^{\varepsilon}}{t^{1/2}K^{3/2}}
\min\left\{1,\frac{10K}{|\tau|}\right\}.
\ena
We also note that by our choice $K$ in (1.5) and
$|\tau|\leq (Mt)^{\varepsilon}C^{-1}\sqrt{NK/M_1}$, we have
\bna
t^{-5/2}\ll\frac{(M t)^{\varepsilon}}{t^{1/2}K^{3/2}}
\min\left\{1,\frac{10K}{|\tau|}\right\}.
\ena

\medskip

{\bf \noindent 5.2. Stationary phase expansion for the $v$-integral.}
Now we will study the integral over $v$ in (5.1). Note that the weight function
restricts the $v$-integral
to a range of length $(Mt)^{\varepsilon}N \zeta/aqKM_1$. Thus for $\zeta<K^{-1}$
we can bound the integral over $v$ trivially to get the bound
$(Mt)^{\varepsilon}t^{-1/2}K^{-5/2}
\left(N/aqM_1\right)^{1/2}.$
Denote by
$\mathcal {I}^{**}(q,m,\tau)$ the integral in (5.1). Then
\bea
\mathcal {I}^{**}(q,m,\tau)=c_1
\left(\frac{aqM_1}{Nt}\right)^{1/2}
\int_{K^{-1}}^1
\int_{\mathbb{R}}g(v)e(f(v))\mathrm{d}v\frac{\mathrm{d}\zeta}{\sqrt{\zeta}}
+O\left(\frac{(Mt)^{\varepsilon}}{t^{1/2}K^{5/2}}
\left(\frac{N}{aqM_1}\right)^{1/2}\right),
\eea
where
\bna
g(v)=\frac{aqM\sqrt{t(t+Kv)}}{N(\zeta M_2-ma)}
U^*\left(\frac{(t+Kv)aqM}{2\pi N(\zeta M_2-ma)}\right)
V\left(\frac{(Kv-\tau)aqM_1}{2\pi N \zeta}\right)V(v)
\ena
and
\bna
f(v)=-\frac{t+Kv}{2\pi}\log\frac{(t+Kv)aqM}{2\pi eN(\zeta M_2-ma)}
+\frac{Kv-\tau}{2\pi}\log \frac{(Kv-\tau)aqM_1}{2\pi eN\zeta}.
\ena
By explicit computations,
\bna
f'(v)
=\frac{K}{2\pi}\log \frac{(Kv-\tau)(\zeta M_2-ma)}{(t+Kv)\zeta M_2}.
\ena
and for $j\geq 2$,
\bna
f^{(j)}(v)=\frac{(-1)^j(j-2)!}{2\pi}\left(\frac{K^j}{(Kv-\tau)^{j-1}}-
\frac{K^j}{(Kv+t)^{j-1}}\right).
\ena
The stationary phase is given by
\bna
v_0=\frac{(t+\tau)M_2 \zeta-\tau ma}{-Kma}.
\ena
In the support of the integral, we have
\bna
g^{(j)}(v)\ll \left(1+\frac{aqKM_1}{N \zeta}\right)^j, \qquad j\geq 0.
\ena
and by the range of $K$,
\bna
f^{(j)}(v)
\asymp \frac{N \zeta}{aqM_1}\left(\frac{aqKM_1}{N \zeta}\right)^{j}, \qquad j\geq 2.
\ena
Moreover, if $v_0\not\in [0.5,3]$, then in the support of the integral we also have
\bna
f'(v)
&=&\frac{K}{2\pi}\log \left(1+\frac{K(v_0-v)}{t+Kv}\right)
-\frac{K}{2\pi}\log \left(1+\frac{K(v_0-v)}{Kv-\tau}\right)\\
&\asymp&K\log \left(1+\frac{K(v_0-v)}{Kv-\tau}\right)
\gg K\min\left\{1,\frac{aqKM_1}{N\zeta}\right\}.
\ena

According as the lower bound of $f'(v)$, we distinguish two cases.

\medskip

{\bf Case a.} $N \zeta/aqKM_1\geq 1$. If $v_0\not\in [0.5,3]$, then the length of the integral is $b-a=1$.
Applying Lemma 2 (1) with
\bna
\Theta_f=\frac{N\zeta}{aqM_1}, \quad \Omega_f=\frac{N\zeta}{aqKM_1}, \quad
\Omega_g=1 \quad \mbox{and}\quad \Lambda=\frac{aqK^2M_1}{N\zeta},
\ena
we obtain
\bna
\int_{\mathbb{R}}g(v)e(f(v))\mathrm{d}v
\ll\frac{1}{K^2}\left(\frac{N}{aqKM_1}\right)^3.
\ena
If $v_0\in [0.5,3]$, then treating the integral as a finite integral
over the range $[0.1,4]$ and applying Lemma 2 (2), it follows
that
\bna
\int_{\mathbb{R}}g(v)e(f(v))\mathrm{d}v=\frac{g(v_0)e\left(f(v_0)+1/8\right)}{\sqrt{f''(v_0)}}
+O\left(\left(\frac{N}{aqK^2M_1}\right)^{3/2}\right).
\ena
Thus for $K$ as in (1.5), we have
\bea
&&\left(\frac{aqM_1}{Nt}\right)^{1/2}
\int_{K^{-1}}^11_{\frac{N \zeta}{aqKM_1}\geq 1}
\int_{\mathbb{R}}g(v)e(f(v))\mathrm{d}v\frac{\mathrm{d}\zeta}{\sqrt{\zeta}}\nonumber\\
&=&
\left(\frac{aqM_1}{Nt}\right)^{1/2}
\int_{K^{-1}}^11_{\frac{N \zeta}{aqKM_1}\geq 1}
\frac{g(v_0)e\left(f(v_0)+1/8\right)}{\sqrt{f''(v_0)}}
\frac{\mathrm{d}\zeta}{\sqrt{\zeta}}
+O\left(\frac{N}{aqK^3M_1\sqrt{t}}\right),
\eea
where $1_{S}=1$ if $S$ is true, and is 0
otherwise, which denotes the characteristic function of the set $S$.

\medskip

{\bf Case b.} $N \zeta/aqKM_1<1$. In this case
$[a,b]=\left[\tau/K-2\pi N\zeta/aqKM_1,
\tau/K+4\pi N\zeta/aqKM_1\right]$ and we apply
Lemma 2 with
\bna
&&\Theta_f=\frac{N\zeta}{aqM_1},\quad
\Omega_f=\frac{N\zeta}{aqKM_1},\quad
\Omega_g=\frac{N\zeta}{aqKM_1} \quad \mbox{and} \quad
\Lambda=K.
\ena
If $v_0\not\in [a,b]$, then
\bna
\int_{\mathbb{R}}g(v)e(f(v))\mathrm{d}v\ll\frac{1}{K^2\Omega_f}.
\ena
If $v_0\in [a,b]$, treating the integral as a finite
integral over $[\tau/K-3\pi N\zeta/aqKM_1,$
$\tau/K+5\pi N\zeta/aqKM_1]$, then
\bna
\int_{\mathbb{R}}g(v)e(f(v))\mathrm{d}v=\frac{g(v_0)e\left(f(v_0)+1/8\right)}{\sqrt{f''(v_0)}}
+O\left(\frac{1}{K^2\Omega_f}+
\frac{1}{K^{3/2}\Omega_f^{1/2}}\right).
\ena
Recall that $\zeta>K^{-1}$. We have $\Omega_f>K^{-1}$ and
the $O$-term above is at most $K^{-1}\sqrt{aqM_1/N \zeta}$. Thus
\bea
&&\left(\frac{aqM_1}{Nt}\right)^{\frac{1}{2}}
\int_{K^{-1}}^11_{\frac{N \zeta}{aqKM_1}< 1}
\int_{\mathbb{R}}g(v)e(f(v))\mathrm{d}v\frac{\mathrm{d}\zeta}{\sqrt{\zeta}}\nonumber\\
&=&
\left(\frac{aqM_1}{Nt}\right)^{1/2}
\int_{K^{-1}}^11_{\frac{N \zeta}{aqKM_1}< 1}
\frac{g(v_0)e\left(f(v_0)+1/8\right)}{\sqrt{f''(v_0)}}
\frac{\mathrm{d}\zeta}{\sqrt{\zeta}}
+O\left(\frac{aqM_1}{KN\sqrt{t}}\right).
\eea
Note that the $O$-terms in (5.2) and (5.4) are dominated by the $O$-term in (5.3).
By (5.2)-(5.4) we obtain
\bea
\mathcal {I}^{**}(q,m,\tau)=c_1
\left(\frac{aqM_1}{Nt}\right)^{1/2}
\int_{K^{-1}}^1
\frac{g(v_0)e\left(f(v_0)+1/8\right)}{\sqrt{f''(v_0)}}
\frac{\mathrm{d}\zeta}{\sqrt{\zeta}}
+O\left(\frac{N}{aqK^3M_1\sqrt{t}}\right).
\eea

Finally we compute the main term. We have
\bna
f(v_0)=-\frac{t+\tau}{2\pi}\log \left(\frac{-(t+\tau)qM}{2\pi eN m}\right), \qquad
f''(v_0)=\frac{(Kma)^2}{2\pi (t+\tau)(\zeta M_2-ma)\zeta M_2}
\ena
and
\bna
g(v_0)
=\frac{aqM}{N}\left(\frac{-t(t+\tau)}{ma(\zeta M_2-ma)}\right)^{1/2}
V\left(\frac{(t+\tau)qM}{-2\pi Nm}\right)U^*\left(\frac{(t+\tau)qM}{-2\pi Nm}\right)
V\left(\frac{\tau}{K}-\frac{(t+\tau)M_2\zeta}{Kma}\right),
\ena
Plugging these into (5.5) we have
\bna
\mathcal {I}^{**}(q,m,\tau)&=&c_2\frac{t+\tau}{K}
\left(\frac{qM}{-mN}\right)^{3/2}
V\left(\frac{(t+\tau)qM}{-2\pi Nm}\right)U^*\left(\frac{(t+\tau)qM}{-2\pi Nm}\right)\\
&&\left(-\frac{(t+\tau)qM}{2\pi eNm}\right)^{-i(t+\tau)}\int_{K^{-1}}^1
V\left(\frac{\tau}{K}-\frac{(t+\tau)M_2\zeta}{Kma}\right)\mathrm{d}\zeta
+O\left(\frac{N}{aqK^3M_1\sqrt{t}}\right)
\ena
for some absolute constant $c_2$.
Extending the integral to the interval $[0,1]$ at a cost of an error term dominated
by the $O$-term in (5.1), we conclude the following.

\begin{lemma}
We have
\bna
\mathcal {J}^{**}(q,m,\tau)=\mathcal {J}_1(q,m,\tau)+
\mathcal {J}_2(q,m,\tau),
\ena
where
\bea
\mathcal {J}_1(q,m,\tau)&=&\frac{c_3}{K\sqrt{t+\tau}}
\left(-\frac{(t+\tau)qM}{2\pi eNm}\right)^{3/2-i(t+\tau)}
V\left(\frac{(t+\tau)qM}{-2\pi Nm}\right)
\nonumber\\
&&U^*\left(\frac{(t+\tau)qM}{-2\pi Nm}\right)\int_0^1V\left(\frac{\tau}{K}-\frac{(t+\tau)M_2\zeta}{Kma}\right)\mathrm{d}\zeta,
\eea
and
\bea
\mathcal {J}_2(q,m,\tau)=\mathcal {J}^{**}(q,m,\tau)-\mathcal {J}_1(q,m,\tau)
=O\left(\mathcal {B}(C,\tau)(Mt)^{\varepsilon}\right),
\eea
where
\bea
\mathcal {B}(C,\tau)=\frac{1}{t^{1/2}K^{3/2}}
\min\left\{1,\frac{10K}{|\tau|}\right\}+
\frac{N^{1/2}}{t^{1/2}K^{5/2}M_1^{1/2}C}.
\eea
\end{lemma}

\medskip

\section{Estimating $\mathcal {S}^{\sharp}(N)$--II}
\setcounter{equation}{0}
\medskip

Denote by $\mathcal {J}_{\ell,J,\pm}(q,m,n_1^2n_2)$ and
$\mathcal {S}_{1,\ell}(N,C,L,J,\pm)$ the contribution of $\mathcal {J}_\ell(q,m,\tau)$
to $\mathcal {J}_{J,\pm}^{*}(q,m,n_1^2n_2)$ in (4.3) and
$\mathcal {S}_{1}^*(N,C,L,J,\pm)$ in (4.5), respectively.

\medskip

\noindent{\bf 6.1. Estimating $\mathcal {S}_{1,1}(N,C,L,J,\pm)$.}
By Cauchy inequality and (2.1), $\mathcal {S}_{1,1}(N,C,L,J,\pm)$ is bounded by
\bea
&&\frac{N^{3/2}}{\sqrt{MM_1^3}}\sum_{0\leq k\leq \log C}
\mathop{\sum\sum}_{n_1^2n_2\sim L}\frac{|\lambda(n_2,n_1)|}{\sqrt{n_2}}
\left|\mathop{\sum_{q=q_0M_2^k \sim C\atop(q_0,M)=1}}_{n_1|q}
\frac{\chi_2(q_0)}{q^{3/2}}\sum_{Q<a\leq q+Q\atop (a,q)=1}\frac{1}{a}\right.\nonumber\\
&&\left.\sum_{1\leq |m|\leq q(Mt)^{1+\varepsilon}/N\atop m\equiv M_2\overline{a}\bmod q}
\overline{\chi_2}(m_0)
\mathscr{B}(n_1,\pm n_2,m,a,q)\mathcal {J}_{1,J,\pm}(q,m,n_1^2n_2)\right|\nonumber\\
&\leq&\sqrt{\frac{N^3L}{M_1^3M}}\sum_{0\leq k\leq \log C}\sqrt{\mathcal {T}(k)},
\eea
where temporarily,
\bna
\mathcal {T}(k)&=&\sum_{n_1}\sum_{n_2}\frac{1}{n_2}W\left(\frac{n_1^2n_2}{L}\right)
\left|\mathop{\sum_{q=q_0M_2^k \sim C\atop(q_0,M)=1}}_{n_1|q}
\frac{\chi_2(q_0)}{q^{3/2}}\sum_{Q<a\leq q+Q\atop (a,q)=1}\frac{1}{a}\right.\nonumber\\
&&\left.\sum_{1\leq |m|\leq q(Mt)^{1+\varepsilon}/N\atop m\equiv M_2\overline{a}\bmod q}
\overline{\chi_2}(m_0)
\mathscr{B}(n_1,\pm n_2,m,a,q)\mathcal {J}_{1,J,\pm}(q,m,n_1^2n_2)\right|^2
\ena
with $W$ a smooth function supported on $[1/2,3]$, which equals 1 on $[1,2]$
and satisfies $W^{(\ell)}(x)\ll_\ell 1$.
Opening the absolute square and interchanging the order of summations we get
\bea
\mathcal {T}(k)
&=&\sum_{n_1\leq \sqrt{3L}}
\mathop{\sum_{q=q_0M_2^k \sim C\atop(q_0,M)=1}}_{n_1|q}
\frac{\chi_2(q_0)}{q^{3/2}}\sum_{Q<a\leq q+Q\atop (a,q)=1}\frac{1}{a}
\sum_{1\leq |m|\leq q(Mt)^{1+\varepsilon}/N\atop m\equiv M_2\overline{a}\bmod q}
\overline{\chi_2}(m_0)\nonumber\\
&&\mathop{\sum_{q'=q_0'M_2^k \sim C\atop(q_0',M)=1}}_{n_1|q'}
\frac{\overline{\chi_2}(q'_0)}{q'^{3/2}}\sum_{Q<a'\leq q'+Q\atop (a',q')=1}\frac{1}{a'}
\sum_{1\leq |m'|\leq q'(Mt)^{1+\varepsilon}/N
\atop m'\equiv M_2\overline{a'}\bmod q'}\chi_2(m_0')
T^*,
\eea
where
\bna
T^*&=&\sum_{n_2}\frac{1}{n_2}W\left(\frac{n_1^2n_2}{L}\right)
\mathcal {J}_{1,J,\pm}(q,m,n_1^2n_2)
\overline{\mathcal {J}_{1,J,\pm}(q',m',n_1^2n_2)}\nonumber\\
&&\mathscr{B}(n_1,\pm n_2,m,a,q)
\overline{\mathscr{B}(n_1,\pm n_2,m',a',q')}.
\ena

Denote $\widehat{q}=q/n_1$.
Then $\mathscr{B}(n_1,n_2,m,a,q)$ in (4.2) is
\bna
\mathscr{B}(n_1,n_2,m,a,q)
=\chi_1(q)S(a\overline{M_1}, n_2\overline{M_1};\widehat{q})
\sideset{}{^*}\sum_{b\bmod M_1}\overline{\chi_1}(m\overline{M_2}-b)
S(\overline{b\widehat{q}}, n_2\overline{\widehat{q}};M_1).
\ena
Applying Poisson summation formula with modulus $\widehat{q}\widehat{q'}M_1$
we obtain
\bea
T^*=\frac{n_1^2}{qq'M_1}\sum_{n_2\in \mathbb{Z}}
\mathscr{C}^*(n_2)\mathcal {I}^*(n_2),
\eea
where
\bea
\mathscr{C}^*(n_2)=
\sum_{c\bmod \widehat{q}\widehat{q'}M_1}
\mathscr{B}(n_1,c,m,a,q)\overline{\mathscr{B}(n_1,c,m',a',q')}
e\left(\frac{n_2 c}{\widehat{q}\widehat{q'}M_1}\right)
\eea
and
\bea
\mathcal {I}^*(n_2)=\int_{\mathbb{R}}W\left(y\right)
\mathcal {J}_{1,J,\pm}(q,m,Ly)
\overline{\mathcal {J}_{1,J,\pm}(q',m',Ly)}
e\left(-\frac{n_2Ly}{qq'M_1}\right)\frac{\mathrm{d}y}{y}.
\eea

\begin{lemma} We have
$\mathcal {I}^*(n_2)$ is arbitrarily small
unless
$|n_2|\leq (Mt)^{\varepsilon} C\sqrt{NKM_1}/L$ and
\bna
\mathcal {I}^*(n_2)\ll (Mt)^{\varepsilon}B^{*}(n_2),
\ena
where $B^{*}(n_2)$ is given by
\bna
B^{*}(n_2)=\left\{\begin{array}{ll}
\frac{N^{1/2}}{tK^{3/2}M_1^{1/2}C},&\mbox{if}\;n_2=0,\\
\\
\frac{N^{1/2}}{tK^{3/2}(|n_2|L)^{1/2}},
&\mbox{if}\;n_2\neq 0.
\end{array}\right.
\ena
\end{lemma}

The following estimate for the character sum $\mathscr{C}^*(n_2)$ was proved in \cite{Munshi32}.

\begin{lemma}
For $n_2\neq 0$, we have
\bna
\mathscr{C}^*(n_2)\ll \widehat{q}\widehat{q'}(\widehat{q},\widehat{q'},n_2)
M_1^{5/2}(M_1,n_2,m\widehat{q}^2-m'\widehat{q'}^2)^{1/2}
\ena
and for $n_2=0$, the sum vanishes unless $\widehat{q}=\widehat{q'}$
(i.e., $q=q'$) in which case
\bna
\mathscr{C}^*(0)\ll \widehat{q}^2R_{\widehat{q}}(a-a')M_1^{5/2}
(M_1,m-m')^{1/2},
\ena
where
$
R_c(u)=\sideset{}{^*}\sum\limits_{\gamma \bmod c}e\left(u \gamma/c\right)
$
is the Ramanujan sum.
\end{lemma}

By (6.2), (6.3) and Lemma 10, we have, up to an arbitrarily small error term,
\bna
\mathcal {T}(k)&\ll&\frac{(Mt)^{\varepsilon}}{M_1C^5}\sum_{n_1\leq \sqrt{3L}}n_1^2
\mathop{\sum_{q=q_0M_2^k\sim C\atop(q_0,M)=1}}_{n_1|q}
\sum_{Q<a\leq q+Q\atop (a,q)=1}\frac{1}{a}
\sum_{1\leq |m|\leq q(Mt)^{1+\varepsilon}/N\atop m\equiv M_2\overline{a}\bmod q}
\mathop{\sum_{q'=q_0'M_2^k\sim C\atop(q_0',M)=1}}_{n_1|q'}
\sum_{Q<a'\leq q'+Q\atop (a',q')=1}\frac{1}{a'}
\nonumber\\
&&
\sum_{1\leq |m'|\leq q'(Mt)^{1+\varepsilon}/N
\atop m'\equiv M_2\overline{a'}\bmod q'}
\sum_{|n_2|\leq (Mt)^{\varepsilon} C\sqrt{NKM_1}/L }
|\mathscr{C}^*(n_2)|B^{*}(n_2).
\ena
Note that for $(q,M_2)=1$, the condition $m\equiv M_2\overline{a}\bmod q$ implies that
$a\equiv \overline{m}M_2\bmod q$.
By Lemmas 10 and 11, the contribution from $k=0$ is
\bea
&&\frac{(Mt)^{\varepsilon}}{M_1C^5}\sum_{n_1\leq \sqrt{3L}}n_1^2
\mathop{\sum_{q\sim C\atop(q,M)=1}}_{n_1|q}
\sum_{1\leq |m|\leq q(Mt)^{1+\varepsilon}/N}
\sum_{Q<a\leq q+Q\atop a\equiv M_2\overline{m}\bmod 1}\frac{1}{a}
\mathop{\sum_{q'\sim C\atop(q',M)=1}}_{n_1|q'}
\nonumber\\
&&\sum_{1\leq |m'|\leq q'(Mt)^{1+\varepsilon}/N}
\sum_{Q<a'\leq q'+Q\atop a'\equiv M_2\overline{m'}\bmod q'}\frac{1}{a'}
\sum_{|n_2|\leq (Mt)^{\varepsilon} C\sqrt{NKM_1}/L }
|\mathscr{C}^*(n_2)|B^{*}(n_2)\nonumber\\
&\ll&\frac{(Mt)^{\varepsilon}}{Q^2M_1C^5}
\frac{N^{1/2}}{tK^{3/2}M_1^{1/2}C}
\sum_{n_1\leq \sqrt{3L}}n_1^2
\mathop{\sum_{q\sim C\atop(q,M)=1}}_{n_1|q}
\sum_{1\leq |m|\leq q(Mt)^{1+\varepsilon}/N}
\sum_{1\leq |m'|\leq q'(Mt)^{1+\varepsilon}/N}\nonumber\\
&&\widehat{q}^2(m-m',\widehat{q})M_1^{5/2}
(M_1,m-m')^{1/2}\nonumber\\
&&+\frac{(Mt)^{\varepsilon}}{Q^2M_1C^5}
\sum_{n_1\leq \sqrt{3L}}n_1^2
\mathop{\sum_{q\sim C\atop(q,M)=1}}_{n_1|q}
\mathop{\sum_{q'\sim C\atop(q',M)=1}}_{n_1|q'}
\sum_{1\leq |m|\leq q(Mt)^{1+\varepsilon}/N}
\sum_{1\leq |m'|\leq q'(Mt)^{1+\varepsilon}/N}\nonumber\\
&&\sum_{1\leq |n_2|\leq (Mt)^{\varepsilon}C\sqrt{NKM_1}/L}
\widehat{q}\widehat{q'}(\widehat{q},\widehat{q'},n_2)
M_1^{5/2}(M_1,n_2)^{1/2}
\frac{N^{1/2}}{tK^{3/2}(|n_2|L)^{1/2}}\nonumber\\
&\ll& \frac{M_1^{5/2}M^2t}{N^{5/2}K^{1/2}}
+\frac{M_1^2M^2t}{N^{3/2}KL}.
\eea
Note that for $k\geq 1$, $m\equiv M_2\overline{a}\bmod q$
implies that $M_2|m$ and $
a\equiv \overline{(m/M_2)} \bmod q/M_2.$
Thus
\bna
\sum_{Q<a\leq q+Q\atop m\equiv M_2\overline{a}\bmod q}\frac{1}{a}=\sum_{i=0}^{M_2-1}
\sum_{Q+iq/M_2<a\leq Q+(i+1)q/M_2\atop
a\equiv \overline{(m/M_2)} \bmod q/M_2}
\frac{1}{a}
=\sum_{i=0}^{M_2-1}\frac{1}{a_i(m,q)}\asymp \frac{M_2}{Q},
\ena
where $a_i(m,q)$ is the unique solution of $a\equiv \overline{(m/M_2)} \bmod q/M_2$
in $Q+iq/M_2<a\leq Q+(i+1)q/M_2$. Bounding similarly as the case
$k=0$, one sees that the contribution from $k\neq 0$ is dominated by
(6.6). Therefore,
\bna
\mathcal {T}(k)\ll \frac{M_1^{5/2}M^2t}{N^{5/2}K^{1/2}}
+\frac{M_1^2M^2t}{N^{3/2}KL}
\ena
and by (6.1) (also recall that $L\leq
N^{1/2+\varepsilon}K^{3/2}M_1^{3/2}$),
\bea
\mathcal {S}_{1,1}(N,C,L,J,\pm)&\ll& \sqrt{\frac{N^3L}{M_1^3M}}\left(
\frac{M_1^{5/4}M\sqrt{t}}{N^{5/4}K^{1/4}}
+\frac{M_1M\sqrt{t}}{N^{3/4}\sqrt{KL}}\right)\nonumber\\
&\ll&(Mt)^{\varepsilon}N^{3/4}(Mt)^{1/2}
\left(\frac{M_1^{1/2}K^{1/2}}{N^{1/4}}+
\frac{1}{M_1^{1/2}K^{1/2}}\right).
\eea

\noindent{\bf 6.2. Bounding $\mathcal {S}_{1,2}(N,C,L,J,\pm)$.}
Applying Cauchy inequality and (2.1), we have
\bea
\mathcal {S}_{1,2}(N,C,L,J,\pm)
\ll \sqrt{\frac{N^3L}{M_1^3M}}\sum_{0\leq k\leq \log C}
\int\limits_{|\tau|\leq (Mt)^{\varepsilon}C^{-1}\sqrt{NK/M_1}}
\sqrt{\mathcal {R}(k,\tau)}
\mathrm{d}\tau,
\eea
where temporarily,
\bna
\mathcal {R}(k,\tau)&=&\sum_{n_1}\sum_{n_2}\frac{1}{n_2}W\left(\frac{n_1^2n_2}{L}\right)
\left|\mathop{\sum_{q=q_0M_2^k\sim C\atop (q_0,M)=1}}_{n_1|q}\frac{\chi_2(q_0)}{q^{3/2}}
\sum_{Q<a\leq q+Q \atop (a,q)=1}\frac{1}{a}
\sum_{1\leq |m|\leq q(Mt)^{1+\varepsilon}/N \atop m\equiv M_2\overline{a}\bmod q}
\overline{\chi_2}(m_0)\right.\\
&&\left.\mathscr{B}(n_1,\pm n_2,m,a,q)
\mathcal {J}_2(q,m,\tau)\right|^2.
\ena

As before, we open the absolute square and interchange the order of summations to get
\bna
\mathcal {R}(k,\tau)
&=&\sum_{n_1\leq \sqrt{3L}}
\mathop{\sum_{q=q_0M_2^k\sim C\atop (q_0,M)=1}}_{n_1|q}\frac{\chi_2(q_0)}{q^{3/2}}
\sum_{Q<a\leq q+Q \atop (a,q)=1}\frac{1}{a}
\sum_{1\leq |m|\leq q(Mt)^{1+\varepsilon}/N \atop m\equiv M_2\overline{a}\bmod q}
\overline{\chi_2}(m_0)\mathcal {J}_2(q,m,\tau)\\
&&\mathop{\sum_{q'=q_0'M_2^k\sim C\atop (q_0',M)=1}}_{n_1|q'}
\frac{\overline{\chi_2}(q_0')}{q'^{3/2}}
\sum_{Q<a'\leq q'+Q \atop (a',q')=1}\frac{1}{a'}
\sum_{1\leq |m'|\leq q'(Mt)^{1+\varepsilon}/N \atop m'\equiv M_2\overline{a'}\bmod q'}
\chi_2(m_0')\overline{\mathcal {J}_2(q',m',\tau)}R^*,
\ena
where
\bna
R^*=\sum_{n_2}\frac{1}{n_2}W\left(\frac{n_1^2n_2}{L}\right)
\mathscr{B}(n_1,\pm n_2,m,a,q)\overline{\mathscr{B}(n_1,\pm n_2,m',a',q')}.
\ena
Applying Poisson summation with modulus $\widehat{q}\widehat{q'}M_1$, we obtain
\bna
R^*
&=&\frac{n_1^2}{qq'M_1}\sum_{n_2\in \mathbb{Z}}
\mathscr{C}^*(n_2)W^{\dag}\left(\frac{n_2L}{qq'M_1},0\right),
\ena
where $\mathscr{C}^*(n_2)$ is defined in (6.4).
By (2.6), the integral is arbitrarily small if
$|n_2|\gg (Mt)^{\varepsilon}C^2M_1/L$.
By (5.7),
\bna
\mathcal {R}(k,\tau)&\ll&(Mt)^{\varepsilon}\frac{\mathcal {B}(C,\tau)^2}{M_1C^5}\sum_{n_1\leq 2C}n_1^2
\mathop{\sum_{q=q_0M_2^k\sim C\atop (q_0,M)=1}}_{n_1|q}
\sum_{Q<a\leq q+Q \atop (a,q)=1}\frac{1}{a}
\sum_{1\leq |m|\leq q(Mt)^{1+\varepsilon}/N \atop m\equiv M_2\overline{a}\bmod q}
\mathop{\sum_{q'=q_0'M_2^k \sim C\atop (q_0',M)=1}}_{n_1|q'}\\
&&\sum_{Q<a'\leq q'+Q \atop (a',q')=1}\frac{1}{a'}
\sum_{1\leq |m'|\leq q'(Mt)^{1+\varepsilon}/N \atop m'\equiv M_2\overline{a'}\bmod q'}
\sum_{|n_2|\leq (Mt)^{\varepsilon}C^2M_1/L}
|\mathscr{C}(n_2)|,
\ena
where $\mathcal {B}(C,\tau)$ is defined in (5.8).
By Lemmas 10 and 11, we have
\bea
R(0,\tau)&\ll&(Mt)^{\varepsilon}\frac{\mathcal {B}(C,\tau)^2}{M_1Q^2C^5}
\sum_{n_1\leq 2C}n_1^2
\mathop{\sum_{q\sim C\atop (q,M)=1}}_{n_1|q}
\sum_{1\leq |m|\leq C(Mt)^{1+\varepsilon}/N}\nonumber
\\&&\sum_{1\leq |m'|\leq C(Mt)^{1+\varepsilon}/N}
\widehat{q}^2\left(\widehat{q},m-m'\right)M_1^{5/2}
(M_1,m-m')^{1/2}\nonumber\\
&&+(Mt)^{\varepsilon}\frac{\mathcal {B}(C,\tau)^2}{M_1Q^2C^5}
\sum_{n_1\leq 2C}n_1^2
\mathop{\sum_{q\sim C\atop (q,M)=1}}_{n_1|q}
\mathop{\sum_{q'\sim C\atop (q',M)=1}}_{n_1|q'}
\sum_{1\leq |m|\leq C(Mt)^{1+\varepsilon}/N}\nonumber
\\&&\sum_{1\leq |m'|\leq C(Mt)^{1+\varepsilon}/N}
\sum_{1\leq |n_2|\leq (Mt)^{\varepsilon}C^2M_1/L}
\widehat{q}\widehat{q'}
\left(\widehat{q},\widehat{q'},n_2\right)
M_1^{5/2}(M_1,n_2)^{1/2}\nonumber\\
&\ll&(Mt)^{\varepsilon}\mathcal {B}(C,\tau)^2
\left(\frac{KM_1^3 Mt}{N^2}+
\frac{KC^3M_1^{7/2} (M t)^2}{N^3L}\right).
\eea
and similarly, the contribution from $k\neq 0$ is dominated by (6.9).
Thus by (6.8),
\bna
\mathcal {S}_{1,2}(N,C,L,J,\pm)
&\ll&\sqrt{\frac{N^3L}{M_1^3M}}
\left(\frac{K^{1/2}M_1^{3/2} (Mt)^{1/2}}{N}+
\frac{K^{1/2}C^{3/2}M_1^{7/4} M t}{N^{3/2}L^{1/2}}\right)\\
&&\times\int\limits_{|\tau|\leq (Mt)^{\varepsilon}C^{-1}\sqrt{NK/M_1}}
\mathcal {B}(C,\tau)
\mathrm{d}\tau,
\ena
where by (5.8)
\bna
\int\limits_{|\tau|\leq (Mt)^{\varepsilon}C^{-1}\sqrt{NK/M_1}}
\mathcal {B}(C,\tau)
\mathrm{d}\tau
\ll\frac{(Mt)^{\varepsilon}}{t^{1/2}K^{1/2}}
\left(1+\frac{N}{C^2K^{3/2}M_1}\right).
\ena
Thus (note that $L\ll N^{1/2+\varepsilon}K^{3/2}M_1^{3/2}$
and $N/(Mt)^{1+\varepsilon}\leq C\leq \sqrt{N/KM_1}$)
\bna
\mathcal {S}_{1,2}(N,C,L,J,\pm)
\ll(Mt)^{\varepsilon}N^{3/4}
\left(K^{3/4}M_1^{3/4}+\frac{\left(Mt\right)^2}{NK^{3/4}M_1^{1/4}}
+\frac{(M t)^{1/2}}{K^{3/4}M_1^{1/2}}
+\frac{M t}{N^{1/4}K^{3/2}M_1^{3/4}}\right),
\ena
where the second term dominates the last two terms by the range of $M_1$ and our choice of $K$ in (1.5).
Therefore,
\bea
\mathcal {S}_{1,2}(N,C,L,J,\pm)
\ll(Mt)^{\varepsilon}N^{3/4}
\left(K^{3/4}M_1^{3/4}+\frac{\left(Mt\right)^2}{NK^{3/4}M_1^{1/4}}\right).
\eea
Under the assumptions $(M t)^{6/5}/(NM_1)^{3/5}\leq
K\leq (Mt)^2/NM_1$, we see that the bound in (6.10) can be controlled by
(6.7).
By (6.7), Lemmas 7 and 8 we conclude that
\bna
\mathcal {S}_{1}(N)
\ll  (Mt)^{\varepsilon}N^{3/4}(Mt)^{1/2}
\left(\frac{M_1^{1/2}K^{1/2}}{N^{1/4}}+
\frac{1}{M_1^{1/2}K^{1/2}}\right).
\ena
Then Proposition 2 follows in view of our choice of $K$ in (1.5).

\medskip

\noindent{\bf 6.3. Proof of Lemma 10.} We follow closely \cite{Munshi32}.
By (4.3) and (6.5), $\mathcal {I}^*(n_2)$ is
\bea
&&\frac{1}{4\pi^2}\int_{\mathbb{R}}\int_{\mathbb{R}}
\left(\frac{NL}{q^3M_1^3}\right)^{-i\tau}
\left(\frac{NL}{q'^3M_1^3}\right)^{i\tau'}
\gamma_{\pm}\left(-\frac{1}{2}+i\tau\right)
\overline{\gamma_{\pm}\left(-\frac{1}{2}+i\tau'\right)}\nonumber\\
&&\mathcal {J}_{1}(q,m,\tau)\overline{\mathcal {J}_{1}(q',m',\tau')}
W_J(\tau)W_J(\tau')
W^{\dag}\left(\frac{n_2L}{qq'M_1},-i(\tau-\tau')\right)\mathrm{d}\tau\mathrm{d}\tau'.
\eea
By (2.6), the integral $W^{\dag}\left(n_2L/qq'M_1,-i(\tau-\tau')\right)$
is negligible if $|n_2|\geq  (Mt)^{\varepsilon} C\sqrt{NKM_1}/L$.
For smaller $|n_2|$, we plug (5.6) into (6.11) to get
\bna
\mathcal {I}^*(n_2)
&=&\frac{|c_3|^2}{4\pi^2K^2}\int_{\mathbb{R}}\int_{\mathbb{R}}
\left(\frac{NL}{q^3M_1^3}\right)^{-i\tau}
\left(\frac{NL}{q'^3M_1^3}\right)^{i\tau'}
\gamma_{\pm}\left(-\frac{1}{2}+i\tau\right)
\nonumber\\
&&\overline{\gamma_{\pm}\left(-\frac{1}{2}+i\tau'\right)}
\left(-\frac{(t+\tau)qM}{2\pi eNm}\right)^{-i(t+\tau)}
\left(-\frac{(t+\tau')q'M}{2\pi eNm'}\right)^{i(t+\tau')}
\\&&H_J(q,m,a,\tau)H_J(q',m',a',\tau')
W^{\dag}\left(\frac{n_2L}{qq'M_1},-i(\tau-\tau')\right)
\mathrm{d}\tau\mathrm{d}\tau',
\ena
where
\bna
H_J(q,m,a,\tau)&=&\frac{1}{\sqrt{t+\tau}}
\left(-\frac{(t+\tau)qM}{2\pi eNm}\right)^{3/2}
V\left(\frac{(t+\tau)qM}{-2\pi Nm}\right)
U^*\left(\frac{(t+\tau)qM}{-2\pi Nm}\right)
W_J(\tau)\\&&
\int_0^1V\left(\frac{\tau}{K}-\frac{(t+\tau)M_2\zeta}{Kma}\right)\mathrm{d}\zeta
\ena
satisfying the bound
\bna
H_J(q,m,a,\tau)\ll t^{-1/2},
\quad
\quad
\frac{\partial}{\partial \tau}H_J(q,m,a,\tau)
\ll \frac{(M t)^{\varepsilon}}{t^{1/2}(1+|\tau|)}.
\ena

For $n_2=0$, by (2.6) we have $W^{\dag}\left(0,-i(\tau-\tau')\right)$
is arbitrarily small if $|\tau-\tau'|\geq (Mt)^{\varepsilon}$.
For $|\tau-\tau'|\leq (Mt)^{\varepsilon}$, we have $W^{\dag}\left(0,-i(\tau-\tau')\right)\ll 1$ and
\bna
\mathcal {I}^*(n_2)
\ll(Mt)^{\varepsilon}\frac{N^{1/2}}{tK^{3/2}M_1^{1/2}C}.
\ena
For $n_2\neq 0$, we apply (2.5) to get
\bna
W^{\dag}\left(\frac{n_2L}{qq'M_1},-i(\tau-\tau')\right)
&=&\frac{c_4}{\sqrt{\tau'-\tau}}
W\left(\frac{(\tau'-\tau)qq'M_1}{2\pi n_2 L}\right)
\left(\frac{(\tau'-\tau)qq'M_1}{2\pi en_2 L}\right)^{i(\tau'-\tau)}\nonumber\\
&&+O\left(\min\left\{\frac{1}{|\tau'-\tau|^{3/2}},
\left(\frac{C^2M_1}{|n_2|L}\right)^{3/2}\right\}\right)
\ena
for some absolute constant $c_4$.
The contribution from the above $O$-term towards $\mathcal {I}^*(n_2)$
is bounded by
\bna
&&\frac{1}{K^2t}\int\limits_{|\tau|\leq 1+2|J|}\int\limits_{|\tau'|\leq 1+ 2|J|}
\min\left\{\frac{1}{|\tau'-\tau|^{3/2}},
\left(\frac{C^2M_1}{|n_2|L}\right)^{3/2}\right\}
\mathrm{d}\tau\mathrm{d}\tau'\nonumber\\
&\ll&(Mt)^{\varepsilon}\frac{N^{1/2}}{tK^{3/2}(|n_2|L)^{1/2}}.
\ena
For the main term, we write by Fourier inversion
\bna
\left(\frac{2\pi n_2L}{(\tau'-\tau)qq'M_1}\right)^{1/2}
W\left(\frac{(\tau'-\tau)qq'M_1}{2\pi n_2 L}\right)
=\int_{\mathbb{R}}W^{\dag}\left(r,\frac{1}{2}\right)
e\left(\frac{(\tau'-\tau)qq'M_1}{2\pi n_2 L}r\right)\mathrm{d}r.
\ena
Then $\mathcal {I}^*(n_2)$ can be written as
\bna
&&\frac{c_5}{K^2}\left(\frac{qq'M_1}{|n_2|L}\right)^{1/2}
\int_{\mathbb{R}}W^{\dag}\left(r,\frac{1}{2}\right)
\int_{\mathbb{R}}\int_{\mathbb{R}}
\gamma_{\pm}\left(-\frac{1}{2}+i\tau\right)
\overline{\gamma_{\pm}\left(-\frac{1}{2}+i\tau'\right)}
H_J(q,m,a,\tau)
\nonumber\\
&&H_J(q',m',a',\tau')\left(\frac{NL}{q^3M_1^3}\right)^{-i\tau}
\left(\frac{NL}{q'^3M_1^3}\right)^{i\tau'}
\left(-\frac{(t+\tau)qM}{2\pi eNm}\right)^{-i(t+\tau)}
\left(-\frac{(t+\tau')q'M}{2\pi eNm'}\right)^{i(t+\tau')}
\\
&&\left(\frac{(\tau'-\tau)qq'M_1}{2\pi en_2 L}\right)^{i(\tau'-\tau)}
e\left(\frac{(\tau'-\tau)qq'M_1}{2\pi n_2 L}r\right)
\mathrm{d}\tau\mathrm{d}\tau'\mathrm{d}r+O\left((Mt)^{\varepsilon}B^{*}(n_2)\right)
\ena
for some absolute constant $c_5$, where for $n_2\neq 0$,
\bna
B^*(n_2)=\frac{N^{1/2}}{tK^{3/2}(|n_2|L)^{1/2}}.
\ena

Note that for $J=0$, we have trivially
$\mathcal {I}^*(n_2)\ll
N^{1/2}/tK^{5/2}(|n_2|L)^{1/2}$
which is dominated by $B^{*}(n_2)$. In the following,
for notational simplicity we only consider the case of $J>0$.
The same analysis holds for $J<0$.
By (2.3), we write
\bea
\mathcal {I}^*(n_2)
&=&\frac{c_5}{K^2}\left(\frac{qq'M_1}{|n_2|L}\right)^{1/2}
\int_{\mathbb{R}}W^{\dag}\left(r,\frac{1}{2}\right)
\int_{\mathbb{R}}\int_{\mathbb{R}}
g(\tau,\tau')
e\left(f(\tau,\tau')\right)
\mathrm{d}\tau\mathrm{d}\tau'\mathrm{d}r\nonumber\\
&&+O\left((Mt)^{\varepsilon}B^{*}(n_2)\right),
\eea
where
\bna
g(\tau,\tau')=\Psi_{\pm}(\tau)\overline{\Psi_{\pm}(\tau')}
H_J(q,m,a,\tau)H_J(q',m',a',\tau')
\ena
and
\bna
2\pi f(\tau,\tau')&=&3\tau\log\left(\frac{\tau}{e\pi}\right)
-3\tau'\log\left(\frac{\tau'}{e\pi}\right)
-\tau\log\left(\frac{NL}{q^3M_1^3}\right)
+\tau'\log\left(\frac{NL}{q'^3M_1^3}\right)\\
&&-(t+\tau)\log\left(-\frac{(t+\tau)qM}{2\pi eNm}\right)
+(t+\tau')\log\left(-\frac{(t+\tau')q'M}{2\pi eNm'}\right)\\
&&+(\tau'-\tau)\log\left(\frac{(\tau'-\tau)qq'M_1}{2\pi en_2 L}\right)+
\frac{(\tau'-\tau)qq'M_1\ell^2}{n_2 L}r.
\ena
For the double integral over $\tau$, $\tau'$ in (6.12), Munshi \cite{Munshi32}
showed that
\bna
\int_{\mathbb{R}}\int_{\mathbb{R}}
g(\tau,\tau')
e\left(f(\tau,\tau')\right)
\mathrm{d}\tau\mathrm{d}\tau'\ll Jt^{-1+\varepsilon}.
\ena
Then using $W^{\dag}\left(r,\frac{1}{2}\right)\ll_j |r|^{-j}$ we obtain
\bna
\mathcal {I}^*(n_2)\ll(Mt)^{\varepsilon}B^{*}(n_2).
\ena
This completes the proof of Lemma 10.

\bigskip

\noindent
{\sc Acknowledgements.}
The author expresses her heartfelt thanks to Roman Holowinsky and Yongxiao Lin
for many valuable
suggestions, Ritabrata Munshi for useful discussions related to his work and she would like to
thank Department of Mathematics, The Ohio State University for
hospitality.
This work is supported by
the National Natural Science Foundation of China (Grant No. 11101239),
Young Scholars Program of Shandong University, Weihai (Grant No. 2015WHWLJH04),
the Natural Science Foundation of Shandong Province (Grant No. ZR2016AQ15)
and a scholarship from the China
Scholarship Council.

\end{document}